\documentclass[a4paper,12pt]{article}
\usepackage[a4paper,left=28mm,right=20mm,top=25mm,bottom=35mm]{geometry}
\usepackage{amsmath,  amsthm, amssymb}
\numberwithin{equation}{section}

\newtheorem{lemma}{Lemma} 
\newtheorem{corollary}{Corollary}

\begin{document}

\author{Ajai Choudhry}
\title{An arbitrary number of squares whose sum,\\ 
on  excluding any one of them, \\is also a square}

\date{}
\maketitle
\abstract{This paper is concerned with the problem of finding $n$ distinct  squares such that, on excluding any one of them, the sum of the remaining $n-1$ squares is  a square. While parametric solutions are known when $n=3$ and $n=4$, when $n > 4$, only a finite number of numerical solutions, found by computer trials, are known. In fact, efforts to find parametric solutions for $n > 4$ have so far been futile. In this paper we describe two methods of obtaining parametric solutions of the problem, and we apply these methods to get several parametric solutions when $n=5, 6, 7$ or $8$. We also indicate how parametric solutions may be obtained for larger values of $n$}. 
\medskip

\medskip

\bigskip

\section{Introduction}\label{intro}

This paper is concerned with the problem of finding $n$ distinct perfect squares such that, on excluding any one of them, the sum of the remaining $n-1$ squares is also a perfect square where $n$ is a given positive integer.

The case $n=3$ dates back to over three centuries with the first numerical solution having been recorded in 1719. Subsequently, several authors gave different methods of solving the problem and a parametric solution was obtained  by Euler in 1772 (as quoted by Dickson \cite[pp. 497--502]{Di}). A number of numerical solutions were listed by Lal and Blundon \cite{LB}, by Leech \cite{Le} and by Spohn \cite{Sp},  and a detailed investigation was done by Bremner \cite{Br} who obtained several parametric solutions.

For $n=4$ also,  parametric solutions were  obtained by Euler as well as several other authors (see  \cite[pp. 502--505]{Di}).    Lagrange\footnote{a twentieth century namesake of the great eighteenth century mathematician}  studied this case in some detail and described a method of constructing many  new solutions \cite{La}.  

Gill \cite[pp.\ 69--76]{Gi} considered the problem of finding the desired $n$ squares for any positive integer $n$ using trigonometric  functions but  his method is too impractical for obtaining actual  solutions  when $n \geq 5$. In fact, Gill himself noted that even for $n=5$, the smallest solution generated by his method would  yield  squares whose roots probably have at least 30 digits and, in fact, they may have more than 60 digits.

A simple  computational method of finding numerical solutions of the problem when $n \geq 5$ has been given by Lagrange  who computed the smallest solutions for $3 \leq n  \leq 8$, and compiled tables of numerical solutions when $n$ is $ 3, 4 $ or $5$. Lagrange stated that the numerical solutions suggest the existence of parametric solutions when $n \geq 5$ but he could not find them.

In this paper we describe two  methods of finding parametric solutions of the problem when $n \geq 5$, and we actually obtain several parametric solutions when $n$ is 5, 6, 7 or 8. 

\section{Preliminary lemmas}\label{prelim}
In this section we give four lemmas that will be used to obtain parametric solutions of our problem. The first lemma given below has been implicitly assumed in \cite{La} --- we give here an explicit proof.
\begin{lemma}\label{necsuffcond} There exist $n$ perfect squares $x_i^2, i=1, \ldots, n$, such that the sum of any $n-1$ of them is a perfect square if and only if the  simultaneous diophantine equations,
\begin{align}
s & = x_1^2+y_1^2=x_2^2+y_2^2=\ldots = x_n^2+y_n^2, \label{diophchn}\\
s & =x_1^2+y_1^2=x_1^2+x_2^2+\cdots + x_n^2, \label{csum}
\end{align}
have a solution in integers.
\end{lemma}

\begin{proof} If there exist  $n$ perfect squares $x_i^2, i=1, \ldots, n$, such that the sum of any $n-1$ of them is a perfect square, then, for each $i, i=1, \ldots, n$, on excluding $x_i^2$, we may write the sum of the remaining $n-1$ squares as $y_i^2$, that is, $\sum_{i=1}^n x_i^2 -x_i^2=y_i^2$. Hence, for each $i$, we have $x_i^2+y_i^2= \sum_{i=1}^n x_i^2$. Denoting the common sum by $s$, we get a solution of the simultaneous diophantine equations \eqref{diophchn} and \eqref{csum}. 

Conversely, if there exists a solution in integers of the simultaneous diophantine equations \eqref{diophchn} and \eqref{csum}, then for each $i$,  $i=1, \ldots, n$,  the sum  $\sum_{i=1}^n x_i^2 -x_i^2 =x_1^2+y_1^2 -x_i^2=y_i^2$. Thus the sum of any $n-1$ of the perfect squares $x_i^2, i=1, \ldots, n$, is also a perfect square. This proves the lemma. \end{proof}

\begin{lemma}\label{soldiophchn} If  $\phi_1(f_1, f_2, g_1, g_2, h_1, h_2)$  and $\phi_2(f_1, f_2, g_1, g_2, h_1, h_2)$ are two functions defined by
\begin{equation}
\begin{aligned}
\phi_1(f_1, f_2, g_1, g_2, h_1, h_2)& = (f_1g_1 + f_2g_2)h_1 + (f_1g_2 - f_2g_1)h_2, \\
\phi_2(f_1, f_2, g_1, g_2, h_1, h_2)& = (-f_1g_2 + f_2g_1)h_1 + (f_1g_1 + f_2g_2)h_2,
\end{aligned}
\label{defphi}
\end{equation}
a solution of the diophantine chain  
\begin{equation}
a_1^2+b_1^2=a_2^2+b_2^2=a_3^2+b_3^2=a_4^2+b_4^2, \label{diophchn4} 
\end{equation}
is given by 
\begin{equation}
\begin{aligned}
(a_1, b_1) & =(\phi_1(p_1, p_2, q_1, q_2, r_1, r_2), \phi_2(p_1, p_2, q_1, q_2, r_1, r_2)), \\
(a_2, b_2) & = (\phi_1(p_1, -p_2, q_1, q_2, r_1, r_2), \phi_2(p_1, -p_2, q_1, q_2, r_1, r_2)), \\
(a_3, b_3) & = (\phi_1(p_1, p_2, q_1, -q_2, r_1, r_2), \phi_2(p_1, p_2, q_1, -q_2, r_1, r_2)), \\
(a_4, b_4) & = (\phi_1(p_1, p_2, q_1, q_2, r_1, -r_2), \phi_2(p_1, p_2, q_1, q_2, r_1, -r_2)), 
\end{aligned}
\label{valabchn4}
\end{equation}
where $p_i, q_i, r_i, i=1, 2$, are arbitrary parameters.

Further, if $\psi_1(e_1, e_2, f_1, f_2, g_1, g_2, h_1, h_2)$  and $\psi_2(e_1, e_2, f_1, f_2, g_1, g_2, h_1, h_2)$ are two functions defined by 
\begin{equation}
\begin{aligned}
\psi_1(e_1, e_2, f_1, f_2, g_1, g_2, h_1, h_2)& =(-e_1f_1g_2 + e_1f_2g_1 - e_2f_1g_1 - e_2f_2g_2)h_1\\
& \quad \quad  + (e_1f_1g_1 + e_1f_2g_2 - e_2f_1g_2 + e_2f_2g_1)h_2, \\
\psi_2(e_1, e_2, f_1, f_2, g_1, g_2, h_1, h_2)& =(e_1f_1g_1 + e_1f_2g_2 - e_2f_1g_2 + e_2f_2g_1)h_1\\
 & \quad \quad + (e_1f_1g_2 - e_1f_2g_1 + e_2f_1g_1 + e_2f_2g_2)h_2,
\end{aligned}
\label{defpsi}
\end{equation}
a solution of the diophantine chain,  
\begin{equation}
a_1^2+b_1^2=a_2^2+b_2^2=\ldots=a_8^2+b_8^2, \label{diophchn8}
\end{equation}
is given by
\begin{equation}
\begin{aligned}
(a_1, b_1) &  =(\psi_1(p_1, p_2, q_1, q_2, r_1, r_2, s_1, s_2), \psi_2(p_1, p_2, q_1, q_2, r_1, r_2, s_1, s_2)), \\
(a_2, b_2) &  =(\psi_1(p_1, -p_2, q_1, q_2, r_1, r_2, s_1, s_2), \psi_2(p_1, -p_2, q_1, q_2, r_1, r_2, s_1, s_2)), \\
(a_3, b_3) &  =(\psi_1(p_1, p_2, q_1, -q_2, r_1, r_2, s_1, s_2), \psi_2(p_1, p_2, q_1, -q_2, r_1, r_2, s_1, s_2)), \\
(a_4, b_4) &  =(\psi_1(p_1, p_2, q_1, q_2, r_1, -r_2, s_1, s_2), \psi_2(p_1, p_2, q_1, q_2, r_1, -r_2, s_1, s_2)), \\
(a_5, b_5) &  =(\psi_1(p_1, p_2, q_1, q_2, r_1, r_2, s_1, -s_2), \psi_2(p_1, p_2, q_1, q_2, r_1, r_2, s_1, -s_2)), \\
(a_6, b_6) &  =(\psi_1(p_1, -p_2, q_1, -q_2, r_1, r_2, s_1, s_2), \psi_2(p_1, -p_2, q_1, -q_2, r_1, r_2, s_1, s_2)), \\
(a_7, b_7) &  =(\psi_1(p_1, -p_2, q_1, q_2, r_1, -r_2, s_1, s_2), \psi_2(p_1, -p_2, q_1, q_2, r_1, -r_2, s_1, s_2)), \\
(a_8, b_8) &  =(\psi_1(p_1, -p_2, q_1, q_2, r_1, r_2, s_1, -s_2), \psi_2(p_1, -p_2, q_1, q_2, r_1, r_2, s_1, -s_2)), 
\end{aligned}
\label{valabchn8}
\end{equation}
where $p_i, q_i, r_i, s_i, i=1, 2$, are arbitrary parameters.
\end{lemma}
\begin{proof} By repeated use of the identity
\begin{equation}
\begin{aligned}
(u_1^2+u_2^2)(v_1^2+v_2^2) & =(u_1v_1-u_2v_2)^2+(u_1v_2+u_2v_1)^2\\
                           &  =(u_1v_1+u_2v_2)^2+(u_1v_2-u_2v_1)^2,
\end{aligned}
\label{compident}
\end{equation}
we may express 
\begin{equation}
(p_1^2+p_2^2)(q_1^2+q_2^2)(r_1^2+r_2^2)
\label{csumchn4}
\end{equation}
 as a sum of two squares in 4 different ways namely, $a_i^2+b_i^2, i=1, \ldots, 4$, where $a_i, b_i$ are defined by \eqref{valabchn4}. Thus for each $i$, the sum  $a_i^2+b_i^2$ is given by \eqref{csumchn4} which proves the first part of the lemma. Similarly, we may express 
\begin{equation}
(p_1^2+p_2^2)(q_1^2+q_2^2)(r_1^2+r_2^2)(s_1^2+s_2^2)
\label{csumchn8}
\end{equation}
as a sum of two squares in 8 different ways namely, $a_i^2+b_i^2, i=1, \ldots, 8$, with $a_i, b_i$  defined by \eqref{valabchn8}, and as before, this establishes  the second part of the lemma which completes the proof. 
\end{proof}

\begin{lemma}\label{solrepeatedsq} If $n \geq 3$, and we define integers $x_i,  i=1, \ldots, n$, in terms of an arbitrary integer parameter $t$, by 
\begin{equation}
\begin{aligned}
x_1=x_2=\dots = x_{n-2}&=8t(t^2 + 1)(t^2 - 1),\\
x_{n-1} &= (t^2 - 1)\{(n - 2)t^4 + (2n - 20)t^2 + n - 2)\},\\
 x_n& = 2t\{(n - 6)t^4 + (2n + 4)t^2 + n - 6\},
\end{aligned}
\label{gennrepeatedsquares}
\end{equation}
the $n$ squares $x_i^2, i=1, \ldots, n$, are such that when any one of them is excluded, the sum of the remaining $n-1$ squares is also a square.

Further, if $n=m^2+1$ where $m$ is an arbitrary positive integer $\geq 2$,  and we define integers $x_i,  i=1, \ldots, n$, in terms of an arbitrary integer parameter $t$, by 
\begin{equation}
x_1=x_2=\dots = x_{n-1}=2t, \quad  x_n  = (n-2)t^2-1,
\label{splnrepeatedsquares}
\end{equation}
the $n$ squares $x_i^2, i=1, \ldots, n$, have the same property.
\end{lemma}
\begin{proof} If we consider the $n$ squares defined by \eqref{gennrepeatedsquares}, on excluding any one of the first $n-2$ squares, the sum of the remaining $n-1$ squares is 
\[(n-3)x_1^2+x_{n-1}^2+x_n^2 = (n - 2)^2(t^2 + 1)^6.\]
Further, on excluding the squares $x_{n-1}^2$ and $x_n^2$ one by one, we get
\[
\begin{aligned}
x_1^2+x_2^2+\cdots+x_{n-2}^2+x_n^2 &=4t^2\{(n + 2)t^4 + (2n - 12)t^2 + n + 2\}^2,\\
x_1^2+x_2^2+\cdots+x_{n-2}^2+x_{n-1}^2 &=(t^2 - 1)^2\{(n - 2)t^4 + (2n + 12)t^2 + n - 2\}^2.
\end{aligned}
\]
This proves that the $n$ squares defined by \eqref{gennrepeatedsquares} have the stated property. 

Next, we consider the $n$ squares defined by \eqref{splnrepeatedsquares}. On excluding any one of the first $n-1$ squares, the sum of the remaining $n-1$ squares is 
\[
(n-2)x_1^2+x_n^2=\{(n-2)t^2  + 1\}^2,\]
while on excluding $x_n^2$, the sum of the remaining squares is $4m^2t^2$. Thus the $n$ squares defined by \eqref{splnrepeatedsquares} also  have the stated property. 
\end{proof}
\begin{corollary}\label{cor} When $n \geq 3$, a solution of the simultaneous diophantine equations \eqref{diophchn} and \eqref{csum}, in terms of an arbitrary parameter $t$, is given by
\begin{equation}
\begin{aligned}
(x_1, y_1) =(x_2, y_2) & =\dots = (x_{n-2}, y_{n-2}) =(8t(t^2 + 1)(t^2 - 1), (n - 2)(t^2 + 1)^3),\\
(x_{n-1}, y_{n-1}) & =((t^2 - 1)((n - 2)t^4 + (2n - 20)t^2 + n - 2)),\\
& \quad \quad  2t((n + 2)t^4 + (2n - 12)t^2 + n + 2)),\\
(x_n, y_n) & =(2t((n - 6)t^4 + (2n + 4)t^2 + n - 6), \\
& \quad \quad (t^2 - 1)((n - 2)t^4 + (2n + 12)t^2 + n - 2)),
\end{aligned}
\label{gennrepeatedchain}
\end{equation}
and further, when $n=m^2+1$ where $m$ is an arbitrary positive integer $\geq 2$, a parametric solution of equations \eqref{diophchn} and \eqref{csum} is given by
\begin{equation}
\begin{aligned}
(x_1, y_1)& =(x_2, y_2)=\dots = (x_{n-1}, y_{n-1}) =(2t, (n - 2)t^2 + 1),\\
(x_n, y_n) & = ((n-2)t^2-1, 2mt).
\end{aligned}
\label{splnrepeatedchain}
\end{equation}
\end{corollary}

\begin{proof} In view of Lemma \ref{necsuffcond}, whenever there exist $n$ squares such that the sum of any $n-1$ of them is a square, there exists a corresponding solution of the simultaneous diophantine equations \eqref{diophchn} and \eqref{csum}. When the $n$ squares are defined by \eqref{gennrepeatedsquares}, we get the value of $s$ from the relation \eqref{csum}, and since, for each $i$, $y_i^2=s-x_i^2$, we readily get the values of $y_i, i=1, \ldots, n$. We thus get the solution \eqref{gennrepeatedchain} of the simultaneous diophantine equations \eqref{diophchn} and \eqref{csum}. Similarly, corresponding to the $n$ squares defined by \eqref{splnrepeatedsquares}, we get the solution \eqref{splnrepeatedchain} of Eqs. \eqref{diophchn} and \eqref{csum}. The solutions given in the corollary can also be readily verified by direct computation.
\end{proof}
 
\begin{lemma}\label{newsol} If $x_i, y_i, i=1, \ldots, n$, is a known solution of the simultaneous diophantine equations \eqref{diophchn} and \eqref{csum}, then a new solution $X_i, Y_i, i=1, \ldots, n$, of the simultaneous  equations \eqref{diophchn} and \eqref{csum} is given by
\begin{equation}
X_i = (n-2)Sx_i-2Py_i, \quad Y_i=2Px_i + (n - 2)Sy_i, \quad i=1, \ldots, n, \label{newsolgen}
\end{equation}
where
\begin{equation}
\begin{aligned}
P&=x_1y_1+x_2y_2+\cdots x_ny_n, \\
S&=x_1^2+x_2^2+\cdots x_n^2.
\end{aligned}
\label{valPS}
\end{equation}
Further, if $i$ and $j$ are any two distinct integers, $ 1 \leq i \leq n$, and $1 \leq j \leq n$, then 
 the ordered pairs $(X_i, Y_i)$ and $(X_j, Y_j)$ are identical if and only the ordered pairs $(x_i, y_i)$ and $(x_j, y_j)$ are identical.
\end{lemma}
\begin{proof} On multiplying the relations \eqref{diophchn} by $g^2+h^2$ and using the first part of identity \eqref{compident}, we get a new solution of \eqref{diophchn} given by
\begin{equation}
X_i = gx_i-hy_i, \quad Y_i=hx_i+gy_i, \quad i=1, \ldots, n. 
\label{valXY}
\end{equation}
Now $(X_i, Y_i), i=1, \ldots, n$, will also satisfy  the second condition  \eqref{csum} if $X_1^2+Y_1^2=\sum_{i=1}^nX_i^2$. This is a quadratic equation in $g$ and $h$ in which the coefficient of $g^2$ vanishes since $x_i, y_i$ is a known solution of Eq. \eqref{csum}. We may thus remove  the factor $h$ after which   the quadratic equation reduces to the linear equation $2Pg  -(n-2)Sh=0$ where $P$ and $S$ are defined by \eqref{valPS}. We thus get $g=(n-2)S, h=2P$, and on inserting these values in \eqref{valXY}, we get the new  solution of the simultaneous equations \eqref{diophchn} and \eqref{csum} given by \eqref{newsolgen}. 

On solving the two linear equations \eqref{newsolgen} for $x_i, y_i$, we get
\begin{equation}
\begin{aligned}
x_i & =((n - 2)SX_i + 2PY_i)/(4P^2+(n - 2)^2S^2),\\
  y_i & = (-2PX_i + (n - 2)SY_i)/(4P^2+(n - 2)^2S^2).
\end{aligned}
	\label{oldsolgen}
\end{equation}
It immediately follows from the relations \eqref{newsolgen} that if the ordered pairs $(x_i, y_i)$ and $(x_j, y_j)$ are identical, then the ordered pairs $(X_i, Y_i)$ and $(X_j, Y_j)$ are also identical, and the converse follows  similarly from the relations \eqref{oldsolgen}. 
\end{proof}

We note that the new solution  \eqref{newsolgen} mentioned in  Lemma \ref{newsol} has been obtained earlier by Lagrange \cite{La} though in a somewhat different manner.    

\section{Two general methods for obtaining solutions in distinct integers of the simultaneous diophantine equations \eqref{diophchn} and \eqref{csum}}\label{general methods}
Lemma 3 gives us $n$ squares, not all distinct, such that the sum of any $n-1$ of them is a square. We will now describe two methods for obtaining solutions of the simultaneous diophantine equations \eqref{diophchn} and \eqref{csum} in which the integers $x_i, i=1, \ldots, n$, are all distinct, and we will thus obtain $n$ distinct squares such that the sum of any $n-1$ of them is a square. 

\subsection{First method}\label{firstmethod}
In the first method, we begin with a parametric  solution $(x_j, y_j), j=1, \ldots, n$, of the simultaneous Eqs. \eqref{diophchn} and \eqref{csum} given by Corollary \ref{cor}, and we assume that  $r$  of the ordered pairs $(x_j, y_j), j=1, \ldots, r$, are identical while the remaining ordered pairs $(x_j, y_j)$ are distinct. Since each $x_j$ appears only as $x_j^2$ both in Eq. \eqref{diophchn} and in Eq. \eqref{csum}, we now replace the ordered pair $(x_r, y_r)$ by $(-x_r, y_r)$ and we still have a solution of these equations. If we now apply Lemma \ref{newsol}, we  obtain a new solution of the simultaneous Eqs. \eqref{diophchn} and \eqref{csum} in which only $r-1$  ordered pairs $(x_j, y_j), j=1, \ldots, r-1$, are identical while the remaining ordered pairs are distinct. Repeating this process $r-1$ times yields a parametric solution of Eqs. \eqref{diophchn} and \eqref{csum} in which all the ordered pairs $(x_j, y_j), j=1, \ldots, n$, are distinct, and we thus get $n$ squares, in parametric terms, such that the sum of any $n-1$ of them is a square. 

If we start with a parametric solution of Eqs. \eqref{diophchn} and \eqref{csum} given by polynomials of degree $n$ and apply Lemma \ref{newsol}, it follows from the relations \eqref{newsolgen} and \eqref{valPS} that the resulting new parametric solution of Eqs. \eqref{diophchn} and \eqref{csum} will, in general, be of degree $3n$. Thus, the process described above yields parametric solutions of our problem of fairly high degrees.

The above process can be significantly improved in the  following manner. Assuming for simplicity that there are $2^m$ identical ordered pairs $(x_j, y_j)$ in the initial known parametric solution of Eqs. \eqref{diophchn} and \eqref{csum} while the remaining ordered pairs $(x_j, y_j)$ are distinct, we change the signs of $x_j$ in $2^{m-1}$ of the identical ordered pairs $(x_j, y_j)$ and then apply Lemma \ref{newsol}. The resulting new solution of Eqs. \eqref{diophchn} and \eqref{csum} will have two sets of identical ordered pairs, each set having $2^{m-1}$ ordered pairs, while the remaining ordered pairs will be different. In each of these two sets, we change the sign of $x_j$ in $2^{m-2}$ of the identical ordered pairs $(x_j, y_j)$, and apply Lemma \ref{newsol} to get a new solution. We repeat this process $m$ times to get a solution of Eqs. \eqref{diophchn} and \eqref{csum} in which all the ordered pairs are distinct.

We will adapt this technique in  Section \ref{parmsols}  to   generate solutions of Eqs. \eqref{diophchn} and \eqref{csum} in which all the ordered pairs $(x_j, y_j)$ are distinct, and we thus obtain parametric solutions of our problem when  $n=5, 6, 7$ or 8.

\subsection{Second method}\label{secondmethod}
In the second method, we will use solutions of the diophantine chains \eqref{diophchn4} and \eqref{diophchn8} given in Lemma \ref{soldiophchn} to construct parametric solutions of Eq. \eqref{diophchn} and we then  choose the parameters so that the  condition \eqref{csum} is also satisfied. We now describe  the method using the diophantine chain \eqref{diophchn8}. 

To obtain a solution of the diophantine chain \eqref{diophchn} for an arbitrary value of $n$, for each $j$, $j=1, \ldots, n$,  we choose a pair of values $a_i, b_i$, defined by \eqref{valabchn8} and $i \in \{1, \ldots, 8\}$. The $n$ ordered pairs $(a_i, b_i)$  need not be distinct. We can assign to  the ordered pair $(x_j, y_j)$  any of the four values  $(\pm a_i, b_i)$ and  $(\pm b_i, a_i)$.  In view of Lemma \ref{soldiophchn}, for each $j$, the sum $x_j^2+y_j^2$ is given by \eqref{csumchn8}.  We thus have  a solution of Eq. \eqref{diophchn}. 

Having obtained a solution of Eq.\eqref{diophchn}, possibly  with some of the ordered pairs $(x_j, y_j)$ being identical,  we  substitute the values of $x_j, y_j, j=1, \ldots, n$, in Eq. \eqref{csum}, and  get an equation that may be considered as a  quadratic equation in the variables $p_1, p_2$, or in the variables $q_1, q_2$, or in $r_1, r_2$, or  in $s_1, s_2$. While, in general,  there is  no easy way to solve this equation, in certain cases, with a suitable   choice of parameters, we may be able to solve it and thus get  a solution of the simultaneous equations \eqref{diophchn} and \eqref{csum}. Sometimes this may yield a solution in which $x_j$, $j=1, \ldots, n$, are all distinct but more often some of the values of $x_j$ may be repeated. In the latter case, we may  apply Lemma \ref{newsol}, as discussed above,   to obtain solutions in which the values of $x_j$ are all distinct. 

As we shall see in Section \ref{parmsols}, the second method often yields parametric solutions  of lower degree as compared to the first method.

\section{$n $ squares whose sum, on excluding any one of them is a square}\label{parmsols}
In the next four subsections we will  obtain parametric solutions of the simultaneous diophantine equations \eqref{diophchn} and \eqref{csum} when $n=5, 6, 7$ and 8, respectively, and thus obtain examples of five, six, seven and eight squares, respectively, such the sum of these squares, on excluding any one of them, is also a square. 

\subsection{Five squares such that the sum of any four is a square}\label{five} 

\subsubsection{Parametric solutions  obtained by the first method} To obtain five squares such that the sum of any four of them is a square, we will use the solution \eqref{splnrepeatedchain} of the simultaneous Eqs. \eqref{diophchn} and \eqref{csum} given in Corolaary \ref{cor}. When $m=2, n=5$, the solution \eqref{splnrepeatedchain} of Eqs. \eqref{diophchn} and \eqref{csum} may we written, on changing the signs of $x_3$ and $x_4$,  as follows:
\begin{equation}
 \begin{aligned}
x_1 & = x_2 = 2t,  \quad  x_3 = x_4  =-2t,    & x_5 & = 3t^2 - 1, \\
y_1 & = y_2   =y_3  =y_4  = 3t^2 + 1,         &  y_5 &= 4t.
\end{aligned}
\label{n5initsolsimple}
\end{equation}

Now on applying Lemma \ref{newsol} we get a solution of Eqs. \eqref{diophchn} and \eqref{csum} in which $X_1 =X_2$ and $X_3=X_4$. We now change the signs of $X_2$ and $X_4$ and rename the $X_i, Y_i$ as $x_i, y_i$, respectively, to get the following solution of Eqs. \eqref{diophchn} and \eqref{csum}:
\begin{equation}
\begin{aligned}
 x_1  & =   -2t(9t^4 - 30t^2 - 7),  &  x_2  & =   -x_1, \\
 x_3  & =   2t(63t^4 + 30t^2 - 1),   & x_4  & =   -x_3, \\
x_5  & =   (3t^2 - 1)(27t^4 - 2t^2 + 3), &  y_1   &=  81t^6 + 165t^4 + 23t^2 + 3, \\
y_2  & =  y_1, &  y_3  & =   81t^6 + 69t^4 + 55t^2 + 3, \\
y_4  & =  y_3, & y_5  & =   4t(45t^4 + 18t^2 + 5).
\end{aligned}
\end{equation}   

On applying Lemma \ref{newsol} once again, we obtain a solution of  Eqs. \eqref{diophchn} and \eqref{csum}  in which the ordered pairs $(x_i, y_i), i=1, \ldots, 5$, are all distinct. For brevity, we give below only the values of $x_i, i=1, \ldots, 5$, of the last solution: 
\begin{equation}
\begin{aligned}
x_1 & = 2t ( 1358127{t}^{16}+2536920{t}^{14}-1378620{t}^{12}  -2647512{t}^{10} \\
& \quad \quad -1562886{t}^{8}-523416{t}^{6}-108028{t}^{4}-
9192{t}^{2}-369 ), \\
x_2 & = 2t ( 1003833{t}^{16}+2256984{t}^{14}+2211948{t}^{12}+
3067848{t}^{10}\\
& \quad \quad +1484550{t}^{8}+503592{t}^{6}+85964{t}^{4}+6456
{t}^{2}+9 ), \\ 
x_3 & = 2t ( 59049{t}^{16}+4706424{t}^{14}+6963084{t}^{12}+
4532328{t}^{10} \\
& \quad \quad+ 1484550{t}^{8}+340872{t}^{6}+27308{t}^{4}+3096
{t}^{2}+153 ), \\ 
x_4 & = 2t ( 2421009{t}^{16}+6700968{t}^{14}+8750268{t}^{12}+
4710744{t}^{10} \\
& \quad \quad+1562886{t}^{8}+294168{t}^{6}+17020{t}^{4}-3480
{t}^{2}-207 ), \\ 
x_5 & = ( 3{t}^{2}-1 )  ( 27{t}^{4}-2{t}^{2}+3 ) 
 ( 19683{t}^{12}+16362{t}^{10} \\
& \quad \quad+34533{t}^{8}+8748{t}^{6}+
3837{t}^{4}+202{t}^{2}+27 ), \\
\end{aligned}
\label{n5solfin}
\end{equation}
where $t$ is an arbitrary parameter.

It follows from Lemma \ref{necsuffcond} that when $x_i, i=1, \ldots, 5$ are defined by \eqref{n5solfin}, the five  squares $x_1^2, \ldots, x_5^2$ are such that their sums,  taken  four at a time, are all squares. As a numerical example, taking $t=2$, we get the following five squares such that the sum of any four of them is a square:
\[
\begin{aligned}
487111462076^2, \quad & 461523666596^2, \quad & 458176368356^2, \\
1238143955524^2, \quad & 501821857691^2. \quad &
\end{aligned}
\]

We can again apply Lemma \ref{newsol} to the last parametric solution of Eqs.  \eqref{diophchn} and \eqref{csum} that we obtained when we get a new solution of these equations  in terms of polynomials of degree 54 in the parameter $t$ and continue the process to get more parametric solutions.

\subsubsection{Parametric solutions  obtained by the second method} We will now apply the second method, described in Section \ref{secondmethod}, to obtain five distinct squares with the desired property 

In view of Lemma \ref{soldiophchn}, when $n=5$, a solution of the   diophantine chain  Eq. \eqref{diophchn}  is given by 
\begin{equation}
(x_i, y_i)=(a_1, b_1), i=1, 2, 3, \;\; (x_4, y_4)=(-a_1, b_1), \;\; (x_5, y_5)=(-a_2, b_2) , \label{n5initsol}
\end{equation}
where $(a_i, b_i), i=1, 2, 3$, are defined by \eqref{valabchn4}. On substituting these values of $x_i, y_i, i=1, \ldots, 5$, in Eq. \eqref{csum} where $n=5$, and transposing all terms to the left-hand side, we get the following condition:
\begin{multline}
\{(4p_1^2 - p_2^2)q_1^2 - 2p_1p_2q_1q_2 - (p_1^2 - 4p_2^2)q_2^2\}r_1^2 - \{2p_1p_2q_1^2 - (6p_1^2 - 6p_2^2)q_1q_2\\
 - 2p_1p_2q_2^2\}r_1r_2 - \{(p_1^2 - 4p_2^2)q_1^2 - 2p_1p_2q_1q_2 - (4p_1^2 - p_2^2)q_2^2\}r_2^2=0. \label{n5qdr1}
\end{multline}
Eq. \eqref{n5qdr1} is a quadratic equation in $r_1$, and its discriminant with respect to $r_1$ namely,
\begin{multline}
16\{(q_1^2 - q_2^2)^2p_1^4 - 4q_1q_2(q_1^2 - q_2^2)p_1^3p_2 - (4q_1^4 + 4q_2^4)p_1^2p_2^2\\
 + 4q_1q_2(q_1^2 - q_2^2)p_1p_2^3 + (q_1^2 - q_2^2)^2p_2^4\}r_2^2,
\label{n5disr1}
\end{multline}
is a quartic function of $p_1$ in which  the coefficient of $p_1^4$ is a perfect square, and hence it can readily be made a perfect square for the following values of $p_1, p_2$ found by following a method described by Fermat \cite[p. 639]{Di}: 
\begin{equation}
p_1=4q_1q_2(q_1 - q_2)(q_1 + q_2),\quad p_2=3(q_1^4 + q_2^4). \label{n5valp12}
\end{equation}
We now obtain two rational solutions of Eq. \eqref{n5qdr1} one of which is as follows:
\begin{equation}
\begin{aligned}
r_1 & =6q_1^5 - 17q_1^4q_2 + 24q_1^3q_2^2 - 12q_1^2q_2^3 - 2q_1q_2^4 + 3q_2^5,\\
r_2 & =3q_1^5 - 2q_1^4q_2 - 12q_1^3q_2^2 + 24q_1^2q_2^3 - 17q_1q_2^4 + 6q_2^5. 
\end{aligned}
\label{n5valr12}
\end{equation}
Thus when $p_i, r_i, i=1, 2$, are given by \eqref{n5valp12} and \eqref{n5valr12}, Eq. \eqref{n5qdr1} is satisfied, and on substituting these values of $p_i, r_i$  in the relations \eqref{n5initsol}, we obtain a solution of the simultaneous diophantine equations \eqref{diophchn} and \eqref{csum} which is given by homogeneous polynomials of degree 10 in the parameters $q_1$ and $q_2$. Denoting  the polynomial $\sum_{j=0}^n c_jq_1^{n-j}q_2^j$ by  $(c_0, c_1, \ldots, c_n)$, this solution may be written briefly as follows:
\begin{equation}
\begin{aligned}
x_1 & =  (-9, 48, -71, 64, -34, 0, 34, -64, 71, -48, 9), \\
x_2 & = (9, 0, -41, 64, -46, 0, 46, -64, 41, 0, -9), \\
x_3 & = (-9, 12, 7, -64, 158, -232, 158, -64, 7, 12, -9),\\
y_1 & = (18, -30, 34, -64, 116, -124, 116, -64, 34, -30, 18),\\
y_2 & = (-18, 54, -98, 80, -52, 44, -52, 80, -98, 54, -18),\\
y_3 & = (18, -48, 94, -128, 116, 0, -116, 128, -94, 48, -18),\\
x_4 &= -x_1, \quad x_5 =-x_2, \quad y_4 = y_1, \quad y_5 = y_2
\end{aligned}
\label{n5soldeg10}
\end{equation}

The  solution \eqref{n5soldeg10} yields five squares with the desired property but only three of the five squares  are distinct since $x_4^2=x_1^2$ and $x_5^2=x_2^2$. Since we have taken $x_4 = -x_1$ and $x_5 =-x_2$, on applying Lemma \ref{newsol} to the solution \eqref{n5soldeg10}, we get a new solution $(X_i, Y_i), i=1, \ldots, 5$, of the  simultaneous diophantine equations \eqref{diophchn} and \eqref{csum}  in which   $X_i, i=1, \ldots, 5$, are all distinct.  

The new solution $(X_i, Y_i), i=1, \ldots, 5$ of Eqs. \eqref{diophchn} and \eqref{csum}  is given by \eqref{newsolgen} where $n=5$,  $P$ and $S$ are defined by the relations \eqref{valPS} and the values of $x_i, y_i$ are given by \eqref{n5soldeg10} in terms of the arbitrary parameters $q_1$ and $q_2$. The values of $X_i, Y_i$ are, in fact,  given by homogeneous polynomials of degree 30 in the parameters $q_1$ and $q_2$. We rename the $X_i$ as $x_i$,  and their values may be written in the aforementioned abbreviated notation as follows:
\begin{align*}
x_1 & =( -5103, 77760, -423711, 1474848, -3856635, 8146752, -14390795, \\
& \quad \quad 21998720,  -30084871, 37826368, -44409783, 48649632, -48301531, \\
& \quad \quad 40374208,  -23382251, 0, 23382251, -40374208, 48301531, -48649632,\\
& \quad \quad 44409783, -37826368,  30084871, -21998720, 14390795, -8146752,\\
 & \quad \quad 3856635, -1474848, 423711, -77760, 5103),\\
x_2 & =( 5103, -11664, -44145, 305856, -881109, 1509552, -1099637,\\ 
 & \quad \quad -2159872,  9288455, -18658448, 25801767, -26143296, 19160459,\\
  & \quad \quad -9562832,  2863723, 0, -2863723, 9562832, -19160459, 26143296,\\
 & \quad \quad  -25801767,  18658448, -9288455, 2159872, 1099637, -1509552,\\
  & \quad \quad 881109,  -305856, 44145, 11664, -5103),
	\end{align*}
	
	\begin{align*}
x_3 & = (-5103, 28188, -76221, 103032, 55017, -668604, 1836067,\\
  & \quad \quad -3199280, 4231697, -5560372, 10712211, -25955064, 55590593,\\
  & \quad \quad	-95585068, 131718715, -146493088, 131718715, -95585068, \\
  & \quad \quad	55590593, -25955064, 10712211, -5560372, 4231697, -3199280,\\
	  & \quad \quad 1836067, 	-668604, 55017, 103032, -76221, 28188, -5103),\\
		x_4 & =(16767, -143856, 616815, -1816992, 4081419, -7243824, 10175323,\\
	 & \quad \quad	-10940032, 7543703, 1045264, -13915641, 27840480, -37484437,\\
	 & \quad \quad 36850000, -23077445, 0, 23077445, -36850000, 37484437, \\
	 & \quad \quad -27840480, 13915641, -1045264, -7543703, 10940032, -10175323,\\
	 & \quad \quad 7243824, -4081419, 1816992, -616815, 143856, -16767),\\
	x_5 & = (-16767, 93312, -252639, 336960, 248517, -2727936, 8526373,\\
		 & \quad \quad	-18083584, 29622377, -38980352, 41242761, -34097856, 20483813,\\
			 & \quad \quad	-7454336, 770245, 0, -770245, 7454336, -20483813,\\
				 & \quad \quad	34097856, -41242761, 38980352, -29622377, 18083584, -8526373, \\
					 & \quad \quad	2727936, -248517, -336960, 252639, -93312, 16767).
		\end{align*}

The above  solution gives five distinct squares $x_i^2, i=1, \ldots, 5$, in terms of polynomials of degree 30, such that the sum of any four of them is a square. 

As a numerical example, taking $q_1=1, q_2=2$, we get, after appropriate scaling, the following five distinct squares such the sum of any four of them is a perfect square:
\[
3023249^2, \quad 1006607^2, \quad 1231825^2, \quad 473569^2, \quad 3426367^2.
\]

\subsection{Six squares such that  the sum of any five is a square} \label{six}
\subsubsection{Parametric solutions  obtained by the first  method} We will now use the first method described in Section \ref{firstmethod} to find six squares such that the sum of any five of them is a square. A solution of the simultaneous equations \eqref{diophchn} and \eqref{csum} obtained from the solution \eqref{gennrepeatedchain} given in Corollary \ref{cor} with $n=6$, and, with the signs of  $x_3, x_4$ suitably changed, may be written as follows:
\begin{equation}
\begin{aligned}
 (x_1, y_1) & = ( 8t(t^4 - 1),  4(t^2 + 1)^3),  \quad  & (x_2, y_2) & =(x_1, y_1), \\
(x_3, y_3) & =(-x_1, y_1), \quad  & (x_4, y_4) & =(-x_1, y_1),  \\ 
(x_5, y_5) & = (4(t^2 - 1)^3, 16t(t^4 + 1)),  &&\\
(x_6, y_6) & =  (32t^3, 4(t^2 - 1)(t^4 + 6t^2 + 1)). && 
\end{aligned}
\label{n6initsol}
\end{equation}

On applying Lemma \ref{newsol} to the solution \eqref{n6initsol}, we get  a new solution $X_i, Y_i, i=1, \ldots, 6$, of the simultaneous equations \eqref{diophchn} and \eqref{csum} which, after renaming $X_i, Y_i$ as $x_i, y_i$, respectively, and changing the signs of $x_1$ and $x_2$, may be written as follows:
\begin{equation}
\begin{aligned}
x_1 & = 16t^3(t^4 - 1)^3, \quad \quad \quad  x_2 = -16t^3(t^4 - 1)^3, \\
  x_3 & = 4t(t^2 - 1)(t^2 + 1)^7, \quad x_4 = -4t(t^2 - 1)(t^2 + 1)^7,\\
	x_5 & = (t^2 - 1)(t^{16} - 4t^{12} - 128t^{10} + 6t^8 - 128t^6 - 4t^4 + 1), \\
	x_6 & = 2t(t^{16} - 36t^{12} - 186t^8 - 36t^4 + 1), \\
	y_1 = y_2 & =  (t^2 + 1)(t^{16} + 16t^{14} + 28t^{12} + 112t^{10}- 58t^8 + 112t^6 \\
		& \quad \quad+ 28t^4 + 16t^2 + 1),\\
	y_3 = y_4 & = (t^2 + 1)(t^{16} + 8t^{14} + 44t^{12} - 8t^{10} + 166t^8 - 8t^6\\
	    & \quad \quad    + 44t^4 + 8t^2 + 1),\\
	 y_5 & = 2t(3t^{16} + 16t^{14} + 52t^{12} - 16t^{10} + 146t^8 - 16t^6\\ 
	    & \quad \quad          + 52t^4 + 16t^2 + 3),\\
	y_6 & = (t^2 - 1)(t^{16} + 16t^{14} + 92t^{12} + 112t^{10} 
		 + 326t^8 + 112t^6\\& \quad \quad   + 92t^4 + 16t^2 + 1)
\end{aligned}
\label{n6newsol1}
\end{equation}
While \eqref{n6newsol1} gives us six squares $x_i^2, i=1, \ldots, 6$,  such that the sum of any five is a square, these squares are not all distinct since $x_1^2=x_2^2$ and $x_3^2=x_4^2$. On applying Lemma \ref{newsol} once again, we get a new parametric solution $X_i, Y_i, i=1,\ldots, 6$, of Eqs. \eqref{diophchn} and \eqref{csum} given by \eqref{newsolgen} where the values of $x_i, y_i$ are given by \eqref{n6newsol1}, the values of $P$ and $S$ are given by \eqref{valPS} and $n=6$. This solution gives $X_i, Y_i,i=1, \ldots, 6$, as  distinct polynomials of degree 54 in the arbitrary parameter $t$ and we thus get six distinct squares with the desired property. We do not give this parametric solution explicitly as it is cumbersome to write.

As a numerical example, taking $t= 2$, we get the following six squares, after removing the common factor 25, such that the sum of any five of them is a square:
\[
\begin{aligned}
252608637530397000^2,  \quad & 15095604154947000^2,  \quad &  492116002633350000^2,\\
669794768570400000^2,  \quad &  37271037420836643^2,  \quad &  43162876561115524^2.
\end{aligned}
\]

\subsubsection{Parametric solutions  obtained by the second method} We will now obtain a simpler parametric solution for six squares such that the sum of any five of them is a square by applying the second method described in Section \ref{secondmethod}. 

In view of Lemma \ref{soldiophchn}, a solution of the diophantine chain \eqref{diophchn} with $n=6$ is given by  
\begin{equation}
\begin{aligned}
(x_1, y_1) & =(a_1, b_1),   \quad  & (x_2, y_2) & =(b_3, a_3),   \quad  &  (x_3, y_3) & =(a_4, b_4), \\
(x_4, y_4) & =(a_5, b_5),   \quad  & (x_5, y_5) & =(a_6, b_6),   \quad  & (x_6, y_6) & =(a_7, b_7),
\end{aligned}
\label{solxyn6}
\end{equation} 
where $a_i, b_i$ are defined by \eqref{valabchn8}. On substituting these values of $x_i, y_i$ in Eq. \eqref{csum} with $n=6$ and transposing all the terms to the left-hand  side, we get the condition
\begin{multline}
(4(p_1q_2 - p_2q_1)^2r_1^2 + 4(p_1q_1 + p_2q_2)^2r_2^2)s_1^2 + (4(p_1q_2 + p_2q_1)(p_1q_1 - p_2q_2)r_1^2\\
 + 4(p_1q_1 + p_1q_2 + p_2q_1 - p_2q_2)(p_1q_1 - p_1q_2 - p_2q_1 - p_2q_2)r_1r_2
 - 4(p_1q_2 + p_2q_1)\\
\times (p_1q_1 - p_2q_2)r_2^2)s_1s_2 + (4(p_1q_1 + p_2q_2)^2r_1^2 + 4(p_1q_2 - p_2q_1)^2r_2^2)s_2^2=0.
 \label{n6qds12}
\end{multline}

Now Eq. \eqref{n6qds12} is a quadratic equation in $s_1$ and it will have a rational solution if its discriminant with respect to $s_1$ is a perfect square. If we take $(q_1, q_2)=(p_1, p_2)$, the discriminant reduces to 
\begin{multline}
 16s_2^2\{4p_1^2p_2^2(p_1^2 - p_2^2)^2r_1^4 + 4p_1p_2(p_1^2 - p_2^2)(p_1^2 + 2p_1p_2 - p_2^2)(p_1^2 - 2p_1p_2 - p_2^2)r_1^3r_2\\
 - (3p_1^8 + 36p_1^6p_2^2 - 30p_1^4p_2^4 + 36p_1^2p_2^6 + 3p_2^8)r_1^2r_2^2 \\
- 4p_1p_2(p_1^2 - p_2^2)(p_1^2 + 2p_1p_2 - p_2^2)(p_1^2 - 2p_1p_2 - p_2^2)r_1r_2^3\\
 + 4p_1^2p_2^2(p_1^2 - p_2^2)^2r_2^4\},
\label{n6diss1}
\end{multline} 
and the aforementioned method of Fermat yields  the following values of $r_1$ and $r_2$ which make it a perfect square:
\begin{equation}
\begin{aligned}
r_1 & = -2p_1p_2(p_1^2 - p_2^2)(p_1^2 + 2p_1p_2 - p_2^2)(p_1^2 - 2p_1p_2 - p_2^2),\\
r_2 & = p_1^8 + 8p_1^6p_2^2 - 2p_1^4p_2^4 + 8p_1^2p_2^6 + p_2^8.
\end{aligned}
\label{n6valr12}
\end{equation}

Thus, with $(q_1, q_2)=(p_1, p_2)$, and $r_1, r_2$ defined by \eqref{n6valr12}, the quadratic equation \eqref{n6qds12} can be solved to get two solutions of which one is given by  the following values of $s_1, s_2$:
\begin{equation}
s_1 = 2p_1p_2(p_1 - p_2)(p_1 + p_2), \quad s_2 =(p_1^2 + p_2^2)^2.
\label{n6vals12}
\end{equation}
We have omitted the second solution as it is cumbersome to write. We now have a solution of the simultaneous diophantine equations \eqref{diophchn} and \eqref{csum}, and we  get six squares $x_i^2, i=1, \ldots, 6$ with the desired property. Since, however,  $(q_1, q_2)=(p_1, p_2)$, we get, $x_1^2=x_5^2$ and $x_3^2=x_4^2$, and hence only four of the six squares are distinct.

To obtain a solution in which all six squares are distinct, we will find another solution of Eq. \eqref{n6qds12}. We now consider  Eq. \eqref{n6qds12} as a quadratic equation in $q_1$ and $q_2$, and when the values of $r_1, r_2$ and $s_1, s_2$ are given by \eqref{n6valr12} and \eqref{n6vals12}, respectively, we know one solution of this equation namely, $(q_1, q_2)=(p_1, p_2)$, and hence we readily obtain a second solution which is as follows:

\begin{equation}
\begin{aligned}
q_1 & = p_1(p_1^{24} + 26p_1^{22}p_2^2 + 184p_1^{20}p_2^4 + 126p_1^{18}p_2^6 + 2105p_1^{16}p_2^8\\
& \quad \quad  - 2972p_1^{14}p_2^{10} + 7288p_1^{12}p_2^{12} - 5284p_1^{10}p_2^{14} + 2435p_1^8p_2^{16}\\
& \quad \quad  - 94p_1^6p_2^{18} + 272p_1^4p_2^{20} + 6p_1^2p_2^{22} + 3p_2^{24}), \\
q_2 & =p_2(3p_1^{24} + 6p_1^{22}p_2^2 + 272p_1^{20}p_2^4 - 94p_1^{18}p_2^6 + 2435p_1^{16}p_2^8\\
& \quad \quad  - 5284p_1^{14}p_2^{10} + 7288p_1^{12}p_2^{12} - 2972p_1^{10}p_2^{14} + 2105p_1^8p_2^{16}\\
& \quad \quad  + 126p_1^6p_2^{18} + 184p_1^4p_2^{20} + 26p_1^2p_2^{22} + p_2^{24}).
\end{aligned}
\label{n6valq12}
\end{equation}

Now Eq. \eqref{csum} is satisfied, and it follows from Lemma \ref{necsuffcond} and the relations \eqref{solxyn6} that the desired  six squares are given  by $(x_1^2, x_2^2, x_3^2, x_4^2, x_5^2, x_6^2)= (a_1^2, b_3^2, a_4^2, a_5^2, a_6^2, a_7^2)$ 
where  $a_i, b_i$ are defined by \eqref{valabchn8} with the values of  $(q_1, q_2)$, $r_1, r_2$ and $s_1, s_2$  being given by \eqref{n6valq12}, \eqref{n6valr12} and \eqref{n6vals12}, respectively. We thus get a parametric solution for the six squares in terms of six distinct polynomials of degree 38 in arbitrary parameters $p_1$ and $p_2$. 
As before, denoting  the polynomial $\sum_{j=0}^n c_jp_1^{n-j}p_2^j$ by  $(c_0,  \ldots, c_n)$, the values of $x_1, \ldots, x_6$, may be written   explicitly as follows: 
\begin{align*}
x_1 & = (2, 0, 128, 0, 782, 0, 2176, 0, 360, 0, 14592, 0, -61960, 0, \\
  & \quad \quad 225920, 0, -214628, 0, 0, 0, 214628, 0, -225920, 0, 61960,\\
 & \quad	\quad 0, -14592, 0, -360, 0, -2176, 0, -782, 0, -128, 0, -2, 0),
\end{align*}

\begin{align*}
x_2 & = (1, 0, 13, 0, 219, 0, 135, 0, 7668, 0, -10524, 0, 99852, 0,  -172868,\\
 & \quad \quad 0,  428622, 0, -591114, 0, 591114, 0, -428622, 0, 172868, 0,\\
 & \quad \quad  -99852, 0, 10524, 0, -7668, 0, -135, 0, -219, 0, -13, 0, -1),\\
x_3 & =(2, 0, -16, 0, -914, 0, -2848, 0, -13880, 0, -33504, 0, -24360,\\
 & \quad \quad  0, 10080, 0, -113284, 0, 0, 0, 113284, 0, -10080, 0, 24360, \\
 & \quad \quad  0, 33504, 0, 13880, 0, 2848, 0, 914, 0, 16, 0, -2, 0),\\
x_4 & = (2, 0, 32, 0, 462, 0, 5824, 0, 2600, 0, 58304, 0, -85064, 0, 127936, \\
& \quad \quad 0,  -154276, 0, 0, 0, 154276, 0, -127936, 0, 85064, 0, -58304,\\
& \quad \quad 0, -2600, 0, -5824,  0, -462, 0, -32, 0, -2, 0),\\
x_5 & = (6, 0, 112, 0, 906, 0, 736, 0, 8600, 0, -16352, 0, -31288, 0,\\
& \quad \quad -35488, 0, 44468, 0, 0, 0, -44468, 0, 35488, 0, 31288, 0,\\
& \quad \quad 16352, 0, -8600, 0, -736, 0, -906, 0, -112, 0, -6, 0),\\
x_6 & = (4, 0, 64, 0, 292, 0, 2048, 0, 4112, 0, 30464, 0, -30512, 0,\\
& \quad \quad 284672, 0, -498184, 0, 938368, 0, -498184, 0, 284672, 0,\\
& \quad \quad -30512, 0, 30464, 0, 4112, 0, 2048, 0, 292, 0, 64, 0, 4, 0),
\end{align*}

As a numerical example taking $(p_1, p_2) =(1, 2)$, we get, on removing the common factor 5,  the six squares,
\[
\begin{aligned} 
 3520435290636^2, \quad  & 3205366606047^2, \quad & 5429263880052^2,\\
 4996634759436^2, \quad &  3039928895652^2, \quad  & 3341350001384^2.
\end{aligned}
\]
whose sums, taken five at a time, are all perfect squares. 

\subsection{Seven squares such that  the sum of any six is a square}\label{seven}

We may obtain  a parametric solution for seven squares by applying the first method  to  the solution \eqref{gennrepeatedchain} of Eqs. \eqref{diophchn} and \eqref{csum}   with $n=7$. Since the values of $x_i$ are the same for $i=1, \ldots, 5$, we change the signs of $x_4, x_5$ and apply Lemma \ref{newsol} to get a new solution $X_i, Y_i, i=1, \ldots, 7$, of Eqs. \eqref{diophchn} and \eqref{csum} in terms of polynomials of degree 18 in which $X_1, X_2, X_3$ are identical and also $X_4, X_5$ are identical. We now change the signs of $X_3$ and $X_5$, apply Lemma \ref{newsol} again and we get a new solution of Eqs. \eqref{diophchn} and \eqref{csum} in terms of polynomials of degree 54. Finally,  on applying Lemma \ref{newsol} a third time (after an appropriate change of sign in the last solution obtained), we get a solution of Eqs. \eqref{diophchn} and \eqref{csum} in distinct polynomials of degree 162, and we thus get seven distinct squares as desired. As this solution is too cumbersome to write, we do not give it explicitly.  

  We will now obtain a simpler parametric solution by applying the second method described in Section \ref{secondmethod}. A solution of the diophantine chain \eqref{diophchn} with $n=7$ is given by
\begin{equation}
\begin{aligned}
(x_1, y_1) & =(a_5, b_5),   \quad  & (x_2, y_2) & =(-a_5, b_5),   \quad  &  (x_3, y_3) & =(a_1, b_1), \\
(x_4, y_4) & =(a_2, b_2),   \quad  & (x_5, y_5) & =(a_4, b_4),   \quad  & (x_6, y_6) & =(a_6, b_6), \\
(x_7, y_7) &=(b_7, a_7),
\end{aligned}
\label{solxyn7}
\end{equation} 
where $a_i, b_i$ are defined by \eqref{valabchn8}. On substituting these values of $x_i, y_i$ in Eq. \eqref{csum} with $n=7$  and transposing all the terms to the left-hand  side, we get the condition
\begin{multline}
(5(p_1q_2 - p_2q_1)^2r_1^2 + (-6p_1^2q_1q_2 - 2p_1p_2q_1^2 + 2p_1p_2q_2^2 + 6p_2^2q_1q_2)r_1r_2
\\ + 5(p_1q_1 + p_2q_2)^2r_2^2)s_1^2 + (-2(p_1q_1 + p_2q_2)(p_1q_2 - p_2q_1)r_1^2\\
 + (-2p_1^2q_1^2 + 2p_1^2q_2^2 + 24p_1p_2q_1q_2 + 2p_2^2q_1^2 - 2p_2^2q_2^2)r_1r_2 \\
+ 2(p_1q_1 + p_2q_2)(p_1q_2 - p_2q_1)r_2^2)s_1s_2 + (5(p_1q_1 + p_2q_2)^2r_1^2 \\
+ (6p_1^2q_1q_2 + 2p_1p_2q_1^2 - 2p_1p_2q_2^2 - 6p_2^2q_1q_2)r_1r_2 + 5(p_1q_2 - p_2q_1)^2r_2^2)s_2^2=0.
\label{n7qds12}
\end{multline}
We now take $(q_1, q_2)=(p_1, p_2)$ when the first term in the coefficient of $s_1^2$ vanishes and we can readily choose nonzero values for $r_1, r_2$ such that the coefficient of $s_1^2$ in Eq.  \eqref{n7qds12} becomes $0$, and we  can solve  Eq.  \eqref{n7qds12} to get nonzero values for $s_1$ and $s_2$. Thus, a solution of Eq. \eqref{n7qds12} with $(q_1, q_2)=(p_1, p_2)$ is given by 
\begin{equation}
\begin{aligned}
r_1 &  = 5(p_1^2 + p_2^2)^2, \quad r_2=8p_1p_2(p_1 - p_2)(p_1 + p_2), \\
s_1 & =25p_1^8 + 164p_1^6p_2^2 + 22p_1^4p_2^4 + 164p_1^2p_2^6 + 25p_2^8, \\
 s_2 & =16p_1p_2(p_1 - p_2)(p_1 + p_2)(p_1^2 - 4p_1p_2 + p_2^2)(p_1^2 + 4p_1p_2 + p_2^2).
\end{aligned}
\label{n7valr12s12}
\end{equation}

We now follow the same procedure as for Eq. \eqref{n6qds12}, and with the values of $r_1, r_2, s_1, s_2$ given by \eqref{n7valr12s12}, we find that Eq. \eqref{n7qds12} is satisfied by the following values of $q_1$ and $q_2$:
\begin{equation}
\begin{aligned}
q_1 & = p_1(78125p_1^{24} + 1399500p_1^{22}p_2^2 + 8937610p_1^{20}p_2^4 + 23564092p_1^{18}p_2^6\\
 & \quad \quad + 78166867p_1^{16}p_2^8 + 3600152p_1^{14}p_2^{10} + 182941900p_1^{12}p_2^{12}\\
  & \quad \quad- 64814760p_1^{10}p_2^{14} + 51621283p_1^8p_2^{16} + 17916188p_1^6p_2^{18}\\
  & \quad \quad+ 14166858p_1^4p_2^{20} + 2174060p_1^2p_2^{22} + 248125p_2^{24}), \\
q_2 & = p_2(248125p_1^{24} + 2174060p_1^{22}p_2^2 + 14166858p_1^{20}p_2^4 + 17916188p_1^{18}p_2^6\\
 & \quad \quad + 51621283p_1^{16}p_2^8 - 64814760p_1^{14}p_2^{10} + 182941900p_1^{12}p_2^{12}\\
  & \quad \quad+ 3600152p_1^{10}p_2^{14} + 78166867p_1^8p_2^{16} + 23564092p_1^6p_2^{18}\\
  & \quad \quad+ 8937610p_1^4p_2^{20} + 1399500p_1^2p_2^{22} + 78125p_2^{24}).\\
\end{aligned}
\label{n7valq12}
\end{equation}

As before,  Eq. \eqref{csum} is now satisfied, and it follows from Lemma \ref{necsuffcond} that we get our seven squares $x_i^2, i=1, \ldots, 7$, where $x_i$ are given by polynomials of degree 38,  in terms of arbitrary parameters $p_1, p_2$, obtained  from  the relations \eqref{solxyn7}. Since, however, $x_2=-x_1$,   our  seven   squares are not distinct and it is only on applying Lemma \ref{newsol} that we get a parametric solution with all seven squares being given by distinct polynomials of degree 114 in terms of the arbitrary parameters  $p_1$ and $p_2$. As these polynomials are too cumbersome to write, we do not give them explicitly. 

A  numerical example of seven distinct squares,  obtained by taking $(p_1, p_2)=(2, 1)$ in the last parametric  solution, is as follows:

\begin{align*}
& 25765736424692698550940334744919576070938381815190280120023040^2, \\
& 6618151217385566375461260055526447245711772332508510442036960^2, \\
& 124934670757222453453275054167136958692122312064551132326904640^2,\\
&  102408517408585640884244377393390888640073807969431822726302964^2,\\
&  55568355668850822237272806729023622683498611696723291768084640^2,\\
 & 165356494031982641014133841865286335701791214944638697401416960^2,\\
 & 86777104483757139311236990507927175848789766997194052355054023^2.
\end{align*}

\subsection{Eight squares such that  the sum of any seven  is a square}\label{eight}
As in Section \ref{seven}, we may begin with the solution \eqref{gennrepeatedchain} of Eqs. \eqref{diophchn} and \eqref{csum}  taking  $n=8$, and apply Lemma \ref{newsol} three times, making appropriate changes of sign at each stage, and thus obtain a solution of Eqs. \eqref{diophchn} and \eqref{csum}  with   $n=8$ in terms of polynomials of degree 162 in the parameter $t$.  As this solution is too cumbersome to write, we do not give it explicitly.

A parametric solution of lower degree may be obtained by the second method by taking $x_i, y_i, i=1, \ldots, 8$, as follows:
\begin{equation}
\begin{aligned}
(x_1, y_1) & =(a_1, b_1),   \quad  & (x_2, y_2) & =(-a_1, b_1),   \quad  &  (x_3, y_3) & =(a_5, b_5), \\
(x_4, y_4) & =(-a_5, b_5),   \quad  & (x_5, y_5) & =(a_2, b_2),   \quad  & (x_6, y_6) & =(a_4, b_4), \\
(x_7, y_7) &=(a_6, b_6),   \quad & (x_8, y_8) &=(b_7, a_7). &
\end{aligned}
\label{solxyn8}
\end{equation} 
where $a_i, b_i$ are defined by \eqref{valabchn8}. With the above values of $x_i, y_i$, we already have a solution of Eq. \eqref{diophchn} with $n=8$, and further, on substituting these values of $x_i, y_i$ in Eq. \eqref{csum} with $n=8$  and transposing all the terms to the left-hand  side, we get the condition
\begin{multline}
(6(p_1q_2 - p_2q_1)^2r_1^2 - 8q_1q_2(p_1 - p_2)(p_1 + p_2)r_1r_2 + 6(p_1q_1 + p_2q_2)^2r_2^2)s_1^2\\
 - 4(p_1q_1 + p_1q_2 + p_2q_1 - p_2q_2)(p_1q_1 - p_1q_2 - p_2q_1 - p_2q_2)r_1r_2s_1s_2\\
 + (6(p_1q_1 + p_2q_2)^2r_1^2 + 8q_1q_2(p_1 - p_2)(p_1 + p_2)r_1r_2 \\
+ 6(p_1q_2 - p_2q_1)^2r_2^2)s_2^2=0. \label{n8qds12}
\end{multline}
A solution of Eq. \eqref{n8qds12}, obtained exactly as in Section \ref{seven}, is as follows:
\begin{equation}
\begin{aligned}
q_1 & =p_1(729p_1^{24} + 11916p_1^{22}p_2^2 + 68162p_1^{20}p_2^4 + 165308p_1^{18}p_2^6\\
 & \quad \quad + 512615p_1^{16}p_2^8 + 324248p_1^{14}p_2^{10} + 968668p_1^{12}p_2^{12}\\
 & \quad \quad + 148568p_1^{10}p_2^{14} + 481719p_1^8p_2^{16} + 168924p_1^6p_2^{18}\\
  & \quad \quad+ 114434p_1^4p_2^{20} + 18668p_1^2p_2^{22} + 2025p_2^{24}), \\
q_2& =p_2(2025p_1^{24} + 18668p_1^{22}p_2^2 + 114434p_1^{20}p_2^4 + 168924p_1^{18}p_2^6\\
 & \quad \quad+ 481719p_1^{16}p_2^8 + 148568p_1^{14}p_2^{10} + 968668p_1^{12}p_2^{12}\\
  & \quad \quad+ 324248p_1^{10}p_2^{14} + 512615p_1^8p_2^{16} + 165308p_1^6p_2^{18} \\
 & \quad \quad+ 68162p_1^4p_2^{20} + 11916p_1^2p_2^{22} + 729p_2^{24}),\\
r_1 & = 6(p_1^2 + p_2^2)^2, \\
r_2& =8p_1p_2(p_1 - p_2)(p_1 + p_2), \\
s_1& =9p_1^8 + 52p_1^6p_2^2 + 22p_1^4p_2^4 + 52p_1^2p_2^6 + 9p_2^8, \\
s_2& =8p_1p_2(p_1 - p_2)(p_1 + p_2)(p_1^2 + 2p_1p_2 - p_2^2)(p_1^2 - 2p_1p_2 - p_2^2),
\end{aligned}
\label{n8valq12r12s12}
\end{equation}
where $p_1$ and $p_2$ are arbitrary parameters.

We thus get a solution of the simultaneous equations  \eqref{diophchn}  and \eqref{csum} with $n=8$  in terms of polynomials of degree 38. Since in this solution $(x_2, y_2)=(-x_1, y_1)$ and $(x_4, y_4) =(-x_3, y_3)$, we apply Lemma \ref{newsol} and obtain a solution in terms of distinct polynomials  of degree 114 in the parameters $p_1$ and $p_2$. While we omit giving this solution explicitly, when we take $(p_1, p_2)=(2, 1)$, we get, after appropriate scaling,  the following eight squares such that the sum of any seven of them is a square:
\begin{align*}
& 1793303948742806004358839863314163172496768^2, \\
& 35647669259187200217596168619979944579248^2, \\
& 1350335627724462794940009188342343291253664^2, \\
& 509318278582178443295965724916222044316304^2, \\
& 2125938053817649628168174016173670414131324^2,  \\
&  863449486628378007179143401686267952168368^2, \\
& 2943648522151467304268140794140107220456896^2, \\
& 70025522883762244048096185376403205296573^2.
\end{align*}

\section{Concluding remarks}
In this paper we have described two methods of finding $n$ squares such the sum of any $n-1$ of them is a square. We applied these methods to obtain $n$ squares, in parametric terms,  with the desired property when $n=5, 6, 7$ or 8. The first method is easy to implement and it can readily be applied to obtain parametric solutions for values of $n > 8$. However, since we will need to apply Lemma \ref{newsol} several times to obtain a aolution in distinct squares,  we will generally get  parametric solutions of our problem in terms of polynomials of very high degree. The second method is more difficult   but  when it can be implemented successfully, it often leads to solutions of lower degree as compared to the first method.

\noindent Ajai Choudhry, 13/4 A Clay Square, Lucknow - 226001, India.

\noindent E-mail address: ajaic203@yahoo.com

\end{document}